\documentclass[11pt]{amsart}
\usepackage{graphicx,dsfont}
\usepackage[dvipdfm, colorlinks, linkcolor=blue, anchorcolor=green, citecolor=red]{hyperref}
\usepackage{mathpazo}
\newtheorem{theorem}{Theorem}[section]
\newtheorem{corollary}[theorem]{Corollary}
\newtheorem{lemma}[theorem]{Lemma}
\newtheorem{proposition}[theorem]{Proposition}
\theoremstyle{definition}
\newtheorem{definition}[theorem]{Definition}

\newtheorem{question}[theorem]{Question}

\begin{document}
\title[Rooted trees with the same plucking polynomial]{Rooted trees with the same plucking polynomial}%
\author{Zhiyun Cheng}%
\address{School of Mathematical Sciences, Laboratory of Mathematics and Complex Systems, Beijing Normal University, Beijing 100875, China}%
\email{czy@bnu.edu.cn}%

\author{Sujoy Mukherjee}
\address{Department of Mathematics, The George Washington University, Washington, DC 20052, USA}
\email{sujoymukherjee@gwu.edu}

\author{J\'{o}zef H. Przytycki}
\address{Department of Mathematics, The George Washington University, Washington, DC 20052, USA, and University of Gda\'nsk}
\email{przytyck@gwu.edu}

\author{Xiao Wang}
\address{Department of Mathematics, The George Washington University, Washington, DC 20052, USA}
\email{wangxiao@gwu.edu}

\author{Seung Yeop Yang}
\address{Department of Mathematics, The George Washington University, Washington, DC 20052, USA}
\email{syyang@gwu.edu}

\subjclass[2010]{05C05; 05C31}%
\keywords{rooted tree; plucking polynomial; Kauffman bracket skein module}%
\begin{abstract}
In this paper we give a sufficient and necessary condition for two rooted trees with the same plucking polynomial. Furthermore, we give a criteria for a sequence of non-negative integers to be realized as a rooted tree.
\end{abstract}
\maketitle
\section{Introduction}
For a rooted tree $T$ embedded on the upper half plane, a new polynomial $Q(T)\in \mathds{Z}[q]$, called the plucking polynomial was recently introduced by the third author in \cite{Prz2016}. If $T$ consists of a single point then $Q(T)=1$. If $|E(T)|\geq1$, the plucking polynomial $Q(T)$ is defined recursively as follows
\begin{center}
$Q(T)=\sum\limits_{v\in l(T)}q^{r(T, v)}Q(T-v)$.
\end{center}
Here $l(T)$ denotes the leaf-set of $T$, $r(T, v)$ equals the number of edges of $T$ on the right side of the unique path connecting $v$ with the root (we assume the root is situated at the origin), and $T-v$ is the subtree of $T$ obtained by deleting $v$ from $T$. See Figure \ref{figure1} for an example of $r(T, v)$. Note that with a given embedding of $T$, if we fix the direction from left to right then we obtain an ordering on $V(T)$. It is easy to observe that $r(T, v)$ is nothing but the number of vertices which are greater than $v$. As an example, the rooted tree described in Figure \ref{figure1} has plucking polynomial $[2]_q[3]_q[5]_q[6]_q$.
\begin{figure}[h]
\centering
\includegraphics{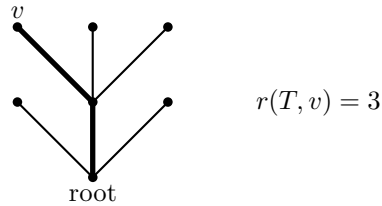}\\
\caption{An example of $r(T, v)$}\label{figure1}
\end{figure}

The definition of the plucking polynomial is motivated by the Kauffman bracket skein modules of 3-manifolds. For an oriented 3-manifold $M$, the \emph{Kauffman bracket skein module} $\mathcal{S}(M)$ \cite{Prz1991} is generated by all isotopy classes of framed links in $M$ and then one takes the quotient by
\begin{enumerate}
  \item skein relation: $[K]=A[K_{\infty}]+A^{-1}[K_0]$,
  \item framing relation: $[K\cup\bigcirc]=(-A^2-A^{-2})[K]$.
\end{enumerate}
Here $[\bigcirc]$ denotes the trivial framed knot and $K, K_{\infty}, K_0$ only differ in a small $D^3$, see Figure \ref{figure2}.
\begin{figure}
\centering
\includegraphics{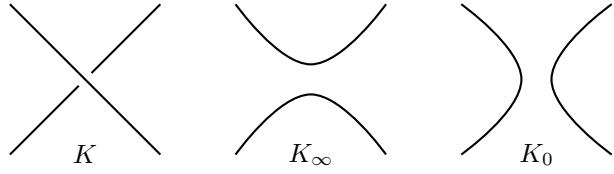}\\
\caption{Local diagrams in skein relation}\label{figure2}
\end{figure}
In \cite{Dab2015}, Dabkowski, Li and the third author studied $(m\times n)$-lattice crossing $L(m, n)$ in the relative Kauffman bracket skein module of $P\times I$, where $P$ denotes an $(m\times n)$ parallelogram with $(2m+2n)$ points on the boundary, see Figure \ref{figure3}.
\begin{figure}[h]
\centering
\includegraphics{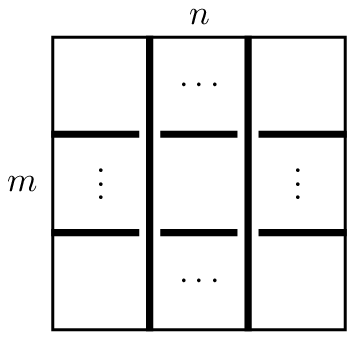}\\
\caption{$L(m, n)$}\label{figure3}
\end{figure}

Roughly speaking, in order to calculate $L(m, n)$ in the relative Kauffman bracket skein module of $P\times I$ one needs to smooth all $mn$ crossing points according to the skein relation $[K]=A[K_{\infty}]+A^{-1}[K_0]$ and then replace each trivial component with $(-A^2-A^{-2})$. The result can be written in the form as
\begin{center}
$[L(m, n)]=\sum\limits_{C\in\text{Cat}_{m, n}}r(C)C$,
\end{center}
where Cat$_{m, n}$ denotes all crossingless connections between the boundary points of $P$ and $r(C)\in \mathds{Z}[A, A^{-1}]$. The explicit formula of $r(C)$ for states $C$ with no returns was investigated in \cite{Dab2015}. After the paper \cite{Dab2015} was finished, it was found that one can construct a rooted tree for each Catalan state $C\in\text{Cat}_{m, n}$ such that the plucking polynomial of this rooted tree is closely related to the coefficient $r(C)$. These relations will be discussed in detail in a sequel paper \cite{Dab2017}.

With this motivation, let us continue with our discussion of plucking polynomial. Although, the definition depends on plane embedding of the rooted tree, the polynomial itself is independent of embedding \cite{Prz2016}, see Section 2. It is natural to ask the following two questions
\begin{enumerate}
  \item For a given polynomial $f(q)\in\mathds{Z}[q]$, does there exist a rooted tree $T$ such that $Q(T)=f(q)$?
  \item When do two rooted trees have the same plucking polynomial?
\end{enumerate}

The first question was answered in \cite{Che2017}, and in this paper we will focus on the second question. Note that if the root of a rooted tree $T$ has only one child, i.e. there is only one edge incident with the root, then by contracting this edge we will obtain a new rooted tree $T'$. Let us call this operation a \emph{destabilization}, and the inverse operation a \emph{stabilization}. According to the definition of plucking polynomial, it is not difficult to see that $T$ and $T'$ have the same plucking polynomial. In other words, stabilization and destabilization both preserve the plucking polynomial. In particular, the plucking polynomial of any 1-ary rooted tree equals 1. We say a rooted tree is \emph{reduced} if the root has more than one child.

At the end of \cite{Che2017}, we introduced the \emph{exchange move} for rooted trees, see Figure \ref{figure4}. More precisely, for a fixed embedded rooted tree $T$ and two vertices $v_1, v_2$, we consider two embedded circles $S_1, S_2$ such that $S_i\cap T=v_i$ $(i=1, 2)$. We use $T_1$ and $T_2$ to denote the subtrees bounded by $S_1$ and $S_2$ respectively. In other words, $T_i$ is a subtree of $v_i$ spanning some children of $v_i$ and all of their descendants ($i=1, 2$). If $|E(T_1)|=|E(T_2)|$ then we switch the positions of $T_1$ and $T_2$. We found in \cite{Che2017} that exchange move preserves plucking polynomial. We asked a natural question whether
if two reduced rooted trees have the same plucking polynomial then one can be obtained from the other by finitely many exchange moves. For all examples illustrated in \cite{Che2017} the answer is yes.
\begin{figure}[h]
\centering
\includegraphics{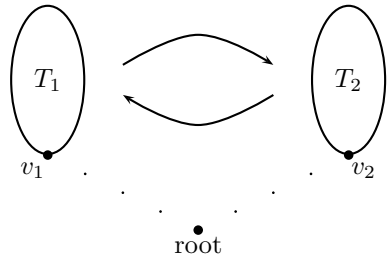}\\
\caption{An exchange move}\label{figure4}
\end{figure}

Following a seminar talk by the first author about plucking polynomial at Peking University in September 2016, Hao Zheng observed the following potential counterexample. Later we will show that this is really a counterexample, i.e. although the two rooted trees in Figure \ref{figure5} have the same plucking polynomial $\frac{[8]_q[11]_q[12]_q[13]_q[14]_q[15]_q[16]_q[17]_q[18]_q[19]_q}{[2]_q^3[3]_q[5]_q[6]_q}$, none of them can be obtained from the other via finitely many exchange moves.
\begin{figure}[h]
\centering
\includegraphics{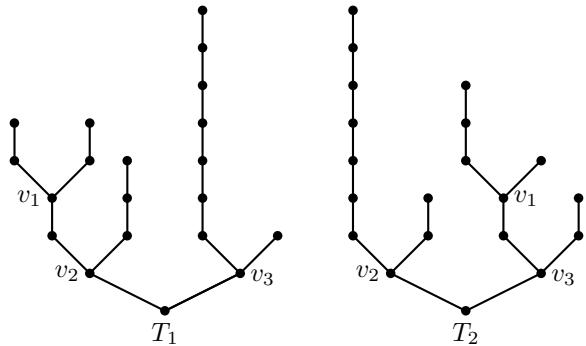}\\
\caption{Two rooted trees with the same plucking polynomial}\label{figure5}
\end{figure}

In order to address the second question mentioned above, we need to introduce a generalized exchange move, called the \emph{permutation move}. The idea of permutation move is as follows. For a rooted tree, we choose $n$ vertices $v_1, \cdots, v_n$ and some families of subtrees rooted at the chosen vertices. Then, for a permutation move, we remove the subtrees and re-glue them back such that the number of edges above the chosen vertices is preserved. The precise definition is as follows.

\begin{definition}
Let us consider $n$ vertices $v_1, \cdots, v_n$ of a rooted tree $T$ and two sequences $\{\alpha_i\}_{0\leq i\leq n}, \{\beta_i\}_{0\leq i\leq n}$ which satisfy
\begin{center}
$0=\alpha_0<\alpha_1<\cdots<\alpha_{n-1}<\alpha_n=\beta_n>\beta_{n-1}>\cdots>\beta_1>\beta_0=0$.
\end{center}
For $v_1$, choose several children of it, say $w_1, \cdots, w_{\alpha_1}$, and draw embedded circles $S^1_i$ $(1\leq i\leq\alpha_1)$ in the plane such that $S^1_i\cap T=v_1$, $w_i$ is located in the interior of $S^1_i$ and other $w_j$ $(j\neq i)$ are located outside of $S^1_i$. We use $T_i$ to denote the subtree of $T$ bounded by $S^1_i$, see Figure \ref{figure6} for an example of $S^1_1$. The other subtrees $T_i$ $(\alpha_1+1\leq i\leq\alpha_n)$ can be defined in the same way. If for any $0\leq i\leq n-1$ and some element $P\in \mathcal{S}_{\alpha_n}$, the symmetric group on the set $\{1, 2, \cdots, \alpha_n\}$, we have
\begin{center}
$\sum\limits_{j=1}^{\alpha_{i+1}-\alpha_i}|E(T_{\alpha_i+j})|=\sum\limits_{j=1}^{\beta_{i+1}-\beta_i}|E(T_{P(\beta_i+j)})|$,
\end{center}
then as illustrated in Figure \ref{figure6}, for all $0\leq i\leq n-1$ we replace $T_{\alpha_i+1}\vee T_{\alpha_i+2}\cdots\vee T_{\alpha_{i+1}}$ with $T_{P(\beta_i+1)}\vee T_{P(\beta_i+2)}\cdots\vee T_{P(\beta_{i+1})}$. We name this operation a \emph{permutation move} on $T$.
\end{definition}
\begin{figure}[h]
\centering
\includegraphics{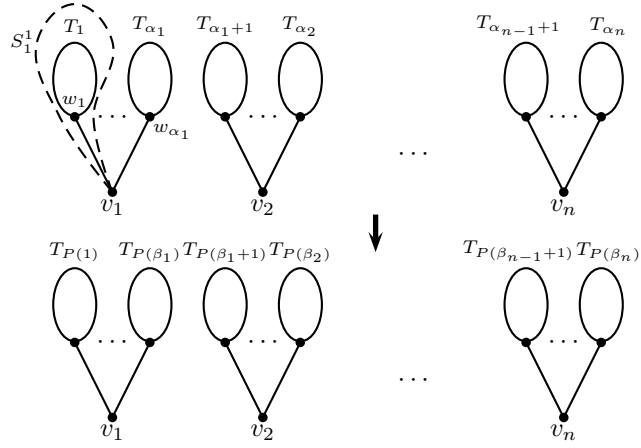}\\
\caption{A permutation move}\label{figure6}
\end{figure}

Clearly a permutation move preserves the plucking polynomial (see Proposition 2.4) and the exchange move described in Figure \ref{figure4} is a special case of permutation move. The main result of this paper is as follows.
\begin{theorem}
Let $T_I$ and $T_{II}$ be two rooted trees. Then $Q(T_I)=Q(T_{II})$ if and only if $T_I$ and $T_{II}$ can be related by a finite number of stabilizations/destabilizations and one permutation move.
\end{theorem}

The rest of this paper is organized as follows. Section 2 reviews some basic properties of plucking polynomial and then points out some other accessible calculation methods of the plucking polynomial. Section 3 is devoted to show that the potential counterexample in Figure \ref{figure5} is indeed a counterexample. Section 4 provides a proof of the main theorem. Finally we revisit the realization problem of the plucking polynomial and give an application of it, which may be of interest to experts in graph theory, algebraic combinatorics, or statistical mechanics.

\section{Some properties of plucking polynomial}
We first recall some standard notations in quantum calculus \cite{Kac2002}. The $q$-analog of $n$, sometimes called the \emph{$q$-bracket} or \emph{$q$-number}, is defined to be $[n]_q=\frac{1-q^n}{1-q}$. Similarly, the \emph{$q$-factorial} can be defined as $[n]_q!=\prod\limits_{i=1}^n[i]_q$. Further the \emph{$q$-binomial coefficients} (also called \emph{Gaussian binomial coefficients}) can be simply expressed as $\binom{m+n}{m,n}_q=\frac{[m+n]_q!}{[m]_q![n]_q!}$. In general, we define the \emph{$q$-multinomial coefficient} as $\binom{n_1+\cdots+n_k}{n_1,\cdots,n_k}_q=\frac{[n_1+\cdots+n_k]_q!}{[n_1]_q!\cdots[n_k]_q!}$.

A crucial observation about the plucking polynomial in \cite{Prz2016} can be described as follows.
\begin{lemma}[\cite{Prz2016}]
Let $T_1, T_2$ be a pair of rooted trees on the upper half plane, and $T_1\bigvee T_2$ the wedge product of $T_1$ and $T_2$ ($T_1$ on the left), then we have
\begin{center}
$Q(T_1\bigvee T_2)=\binom{|E(T_1)|+|E(T_2)|}{|E(T_1)|, |E(T_2)|}_qQ(T_1)Q(T_2)$.
\end{center}
\end{lemma}

Since $\binom{n_1+\cdots+n_k}{n_1,\cdots,n_k}_q=\binom{n_{k-1}+n_k}{n_{k-1}, n_k}_q\binom{n_{k-2}+n_{k-1}+n_{k}}{n_{k-2}, n_{k-1}+n_k}_q\cdots\binom{n_1+\cdots+n_k}{n_1, n_2+\cdots+n_k}_q$, repeating the lemma above $(k-1)$ times one can easily conclude the following result.
\begin{corollary}[\cite{Prz2016}]
Let $T_1, \cdots, T_k$ be $k$ rooted trees on the upper half plane. Denote the wedge product of $T_1, \cdots, T_k$ by $\bigvee\limits_{i=1}^kT_i$ ($T_i$ is on the left of $T_{i+1}$ ). Then
\begin{center}
$Q(\bigvee\limits_{i=1}^kT_i)=\binom{\sum\limits_{i=1}^k|E(T_i)|}{|E(T_1)|, \cdots, |E(T_k)|}_q\prod\limits_{i=1}^kQ(T_i)$.
\end{center}
\end{corollary}

The recursive definition of the plucking polynomial stated in the beginning of Section 1 is inconvenient to use to calculate the plucking polynomial. In general, one needs to pluck out the leaves of a rooted tree one by one and then calculate the plucking polynomial of the new rooted tree with fewer edges. Corollary 2.2 gives us a more convenient way to calculate the plucking polynomial of rooted trees. For a rooted tree $T$ and a fixed vertex $v$, we use the notation $d(v)$ to refer to the number of descendants of $v$. For example, if $v=r$, the root of $T$, then $d(r)=|V(T)|-1=|E(T)|$. Denote all children of $v$ by $v_1, \cdots, v_k$, we associate a \emph{Boltzmann weight} $W(v)$ with $v$, which is defined by
\begin{center}
$W(v)=\binom{d(v)}{d(v_1)+1, \cdots, d(v_k)+1}_q$.
\end{center}
Note that here we have the equation $d(v)=\sum\limits_{i=1}^k(d(v_i)+1)$. Now we have the following state product formula of the plucking polynomial.
\begin{proposition}[State product formula \cite{Prz2016}]
$Q(T)=\prod\limits_{v\in V(T)}W(v)$.
\end{proposition}

It immediately follows that plucking polynomial does not depend on the embedding type. In other words, it is an invariant of rooted trees. On the other hand, since the plucking polynomial can be written as the product of some $q$-multinomial coefficients and each $q$-multinomial coefficient can be written as the product of some $q$-binomial coefficients, we conclude that plucking polynomial of rooted trees can be written as the product of some $q$-binomial coefficients. Based on this fact, in \cite{Che2017} we gave a complete answer to the first question mentioned in Section 1. To address the second question, we need to simplify the calculation formula of plucking polynomial once again.

\begin{proposition}
Let $T$ be a rooted tree and $r$ the root of it, then we have $Q(T)=\frac{[d(r)]_q!}{\prod\limits_{v\in V(T)\backslash{\{r\}}}[d(v)+1]_q}$.
\end{proposition}
\begin{proof}
According to Proposition 2.3, the plucking polynomial $Q(T)$ has the form $\prod\limits_{v\in V(T)}W(v)$. Let $v(\neq r)$ be a vertex of $T$ and $w$ its ancestor. Note that the numerator of $W(v)$ equals $[d(v)]_q!$, and the denominator of $W(w)$ has a factor $[d(v)+1]_q!$. After canceling $[d(v)]_q!$ for all non-root vertices the result follows.
\end{proof}

Proposition 2.4 motivates us to consider the multiset $D(T)=\{d(v)|v\in V(T)\}$\footnote{Equivalently, instead of the set $D(T)$, one can also consider the generating function $\sum\limits_{i}c_ix^i$, where $c_i$ denotes the multiplicity of $i$ in $D(T)$.}. According to Proposition 2.4, the plucking polynomial $Q(T)$ is determined by the set $D(T)$. The following proposition tells us that for reduced trees $Q(T)$ and $D(T)$ are essentially equivalent.

\begin{proposition}
Assume $T_1$ and $T_2$ are two reduced rooted trees, then $Q(T_1)=Q(T_2)$ if and only if $D(T_1)=D(T_2)$.
\end{proposition}
\begin{proof}
It suffices to prove that if $Q(T_1)=Q(T_2)$ then $D(T_1)=D(T_2)$. Assume $D(T_1)=\{a_1, \cdots, a_n\}$ and $D(T_2)=\{b_1, \cdots, b_m\}$, where $a_1\geq a_2\geq\cdots\geq a_n$ and $b_1\geq b_2\geq\cdots\geq b_m$. Since $T_1$ and $T_2$ are both reduced, we observe that $d(r_1)=a_1>a_2+1, d(r_2)=b_1>b_2+1$ and $n=a_1+1, m=b_1+1$. Here $r_i$ denotes the root of $T_i$ $(i=1, 2)$. It follows that
\begin{center}
$Q(T_1)=\frac{[a_1]_q!}{\prod\limits_{i=2}^{n}[a_i+1]_q}=\frac{\prod\limits_{i=1}^{a_1}(1-q^i)}
{\prod\limits_{i=2}^{n}(1-q^{a_i+1})}$ and $Q(T_2)=\frac{[b_1]_q!}{\prod\limits_{i=2}^{m}[b_i+1]_q}=\frac{\prod\limits_{i=1}^{b_1}(1-q^i)}
{\prod\limits_{i=2}^{m}(1-q^{b_i+1})}$.
\end{center}
It is clear that $e^{\frac{2\pi i}{a_1}}$ is a root of $Q(T_1)$ with the minimal argument and $e^{\frac{2\pi i}{b_1}}$ is a root of $Q(T_2)$ with the minimal argument. Since $Q(T_1)=Q(T_2)$, we must have $a_1=b_1$, therefore $n=a_1-1=b_1-1=m$.

Now, as we have $\prod\limits_{i=2}^{n}[a_i+1]_q=\prod\limits_{i=2}^{n}[b_i+1]_q$, it obviously follows that $a_i=b_i$ for all $2\leq i\leq n$.
\end{proof}

We remark that here we only mentioned some basic properties of the plucking polynomial. For some other properties, for example the unimodality of its coefficients (see \cite{Che2018}) or the connection between the plucking polynomial and homological algebra, the readers are referred to \cite{Prz2016} for more details.

\section{The exchange move is not sufficient}
In this section we show that the two rooted trees depicted in Figure \ref{figure5} cannot be connected by exchange moves, although they have the same plucking polynomial. In other words, exchange move is not sufficient to connect all pairs of rooted trees with the same plucking polynomial.

First we notice that one can only obtain finitely many rooted trees from the rooted tree $T_1$ in Figure \ref{figure5} via exchange moves. By comparing them with $T_2$ one will find that $T_2$ is different from all of them.

Recall that in Section 2 we associate a Boltzmann weight $W(v)$ with each vertex $v$, which is defined by
\begin{center}
$W(v)=\binom{d(v)}{d(v_1)+1, \cdots, d(v_k)+1}_q$.
\end{center}
Let $U(v)$ denote the unordered $k$-tuple $(d(v_1)+1, \cdots, d(v_k)+1)$, and $U(T)$ the multiset $\{U(v)\}_{v\in V(T)}$. Consider $v_1, v_2$ in Figure \ref{figure4}. If $U(v_1)=(a_1, \cdots, a_m, b_1, \cdots, b_n)$, $U(v_2)=(c_1, \cdots, c_s, d_1, \cdots, d_t)$, and $\sum\limits_{i=1}^nb_i=\sum\limits_{j=1}^sc_j$, then after the exchange move we have $U(v_1)=(a_1, \cdots, a_m, c_1, \cdots, c_s)$, $U(v_2)=(b_1, \cdots, b_n, d_1, \cdots, d_t)$ and all other tuples in $U(T)$ are preserved.

For the two rooted trees $T_1, T_2$ in Figure \ref{figure5}, we have
\begin{center}
$U(T_1)=\{(9, 10), (3, 6), (1, 7), (2, 2), (6), (5)\times2, (4), (3), (2)\times2, (1)\times4, (0)\times5\}$
\end{center}
and
\begin{center}
$U(T_2)=\{(9, 10), (2, 7), (2, 6), (1, 3), (6), (5)\times2, (4), (3), (2)\times2, (1)\times4, (0)\times5\}$.
\end{center}
A key observation is, although many rooted trees can be obtained from $T_1$ via exchange moves, most of them have the same set $U$ as $T_1$. The only exception is
\begin{center}
$U=\{(3, 6, 10), (9), (1, 7), (2, 2), (6), (5)\times2, (4), (3), (2)\times2, (1)\times4, (0)\times5\}$.
\end{center}
It follows that $T_2$ cannot be obtained from $T_1$ by exchange moves.

One can directly show that essentially there exist only four different rooted trees that can be obtained from $T_1$ by exchange moves, see Figure \ref{figure10}. Here we illustrate how one rooted tree can be obtained from another by exchange moves. Obviously $T_2$ is not one of them.
\begin{figure}[h]
\centering
\includegraphics{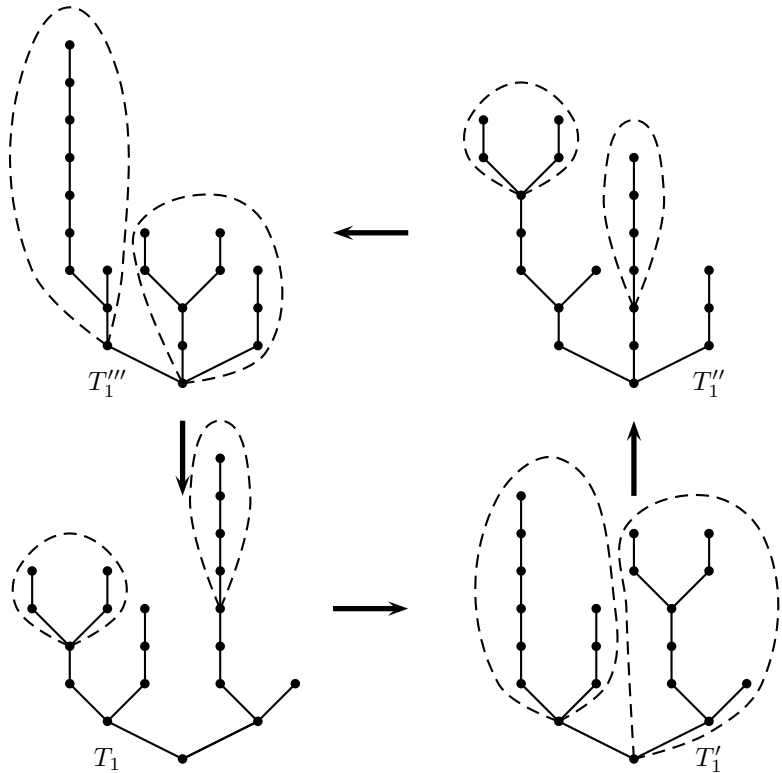}\\
\caption{Rooted trees obtained from $T_1$ by exchange moves}\label{figure10}
\end{figure}

Before ending this section, we would like to remark that one can find some other pairs of ``smaller" rooted trees which have the same plucking polynomial but cannot be connected by exchange moves. For example, the two rooted trees $T_3$ and $T_4$ described in Figure \ref{figure11} have the same plucking polynomial. Similar as above one can check that they are not related by exchange moves. Note that $|E(T_3)|=|E(T_4)|=18<19=|E(T_1)|=|E(T_2)|$. These two rooted trees $T_3$ and $T_4$ can be regarded as a reduced version of Hao Zheng's $T_1, T_2$ described in Figure \ref{figure5}.
\begin{figure}[h]
\centering
\includegraphics{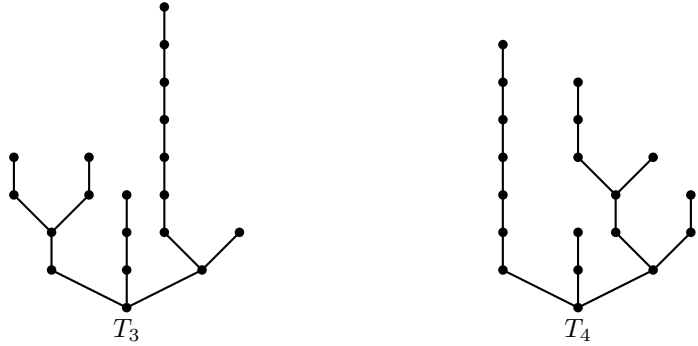}\\
\caption{Another counterexample with fewer edges}\label{figure11}
\end{figure}

\section{The proof of Theorem 1.2}
Now we give a proof of Theorem 1.2.
\begin{proof}
With some destabilization (if necessary) we may assume that $T_I, T_{II}$ are both reduced. By Proposition 2.5 we deduce that $D(T_I)=D(T_{II})$. In particular, $T_I$ and $T_{II}$ have the same number of edges. Now it is sufficient to prove that $T_I$ and $T_{II}$ are related by one permutation move.

The proof goes by induction on $|E(T_I)|(=|E(T_{II})|)$. When $|E(T_I)|=1, 2, 3, 4$, there do not exist a pair of distinct rooted trees with the same plucking polynomial. The first pair of rooted trees appear when $|E(T_I)|=|E(T_{II})|=5$, see Figure \ref{figure7}. It is easy to see that the second tree can be obtained from the first tree by one exchange move, which interchanges the places of the subtrees bounded by the dashed curves.
\begin{figure}[h]
\centering
\includegraphics{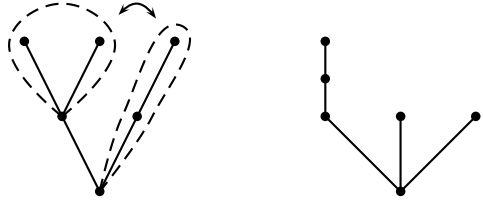}\\
\caption{Two 5-edge rooted trees with the same plucking polynomial}\label{figure7}
\end{figure}

Now we assume that any two reduced rooted trees can be transformed into each other by one permutation move if they have $(k-1)$ edges and the same plucking polynomial. It is sufficient to prove that the statement is still correct for the case $|E(T_{I})|=|E(T_{II})|=k$. Suppose that $T_I, T_{II}$ are two reduced rooted trees with the same plucking polynomial and $|E(T_I)|=|E(T_{II})|=k$. We need to prove that $T_I$ and $T_{II}$ can be related by one permutation move.

Denote the roots of $T_I, T_{II}$ by $r_1, r_2$ respectively. Let us assume that $D(T_I)=D(T_{II})=\{a_1, a_2, \cdots, a_{k+1}\}$, where $a_1\geq a_2\geq\cdots\geq a_{k+1}$. It is clear that $a_1=d(r_1)=d(r_2)=k$. Choose vertices $u_1, u_2$ in $T_I, T_{II}$ respectively such that $d(u_1)=d(u_2)=a_2$. Then $u_i$ must be a child of $r_i$ $(i=1, 2)$. We use $e_{i}$ to denote the edge between $r_i$ and $u_i$ $(i=1, 2)$, and by taking edge contractions on $e_1$ and $e_2$ we will obtain two new rooted trees $T_I'$ and $T_{II}'$. Note that for these two new rooted trees we have $D(T_I')=D(T_{II}')=\{a_1-1, a_3, \cdots, a_{k+1}\}$, and therefore $T_I'$ and $T_{II}'$ have the same plucking polynomial. By the induction assumption, it follows that $T_{II}'$ can be obtained from $T_I'$ by making one permutation move. Since there is a one-to-one correspondence between $E(T_I)\backslash \{e_1\}$ and $E(T_I')$, $V(T_I)\backslash \{u_1\}$ and $V(T_I')$, we will use the same notation to denote an edge (vertex) in $T_I$ and its corresponding edge (vertex) in $T_I'$. In particular, the roots of $T_I'$ and $T_{II}'$ are still denoted by $r_1$ and $r_2$.

If $r_1$ coincides with some element in $\{v_1, \cdots, v_n\}$ (see Figure \ref{figure6}), without loss of generality, let us assume that $v_1=r_1$. In $T_I'$, denote the children of $v_1$ involved in the permutation move by $w_1, \cdots, w_{\alpha_1}$. Then in $T_I$, some elements of $\{w_1, \cdots, w_{\alpha_1}\}$ are adjacent to $u_1$ and others are adjacent to $r_1$. Without loss of generality, we assume $w_1, \cdots, w_\gamma$ are adjacent to $u_1$ and $w_{\gamma+1}, \cdots, w_{\alpha_1}$ are adjacent to $r_1$, where $0\leq \gamma\leq\alpha_1$. As before, for each $i=1, 2, \cdots, \alpha_1$ we still use $T_i$ to denote the subtree involved in the permutation move which contains the vertex $w_i$. After the permutation move, the places of $\bigvee\limits_{i=1}^{\alpha_1}T_i$ will be occupied by $\bigvee\limits_{i=1}^{\beta_1}T_{P(i)}$, where $P$ is a permutation of $\{1, \cdots, \alpha_n\}$; see Figure \ref{figure6}. Similarly, in the original rooted tree $T_{II}$, some of $\{T_{P(1)}, \cdots, T_{P(\beta_1)}\}$ are attached to $u_2$ and others are attached to $r_2$. Let us assume $\bigvee\limits_{i=1}^{\delta}T_{P(i)}$ are attached to $u_2$ and $\bigvee\limits_{i=\delta+1}^{\beta_1}T_{P(i)}$ are attached to $r_2$.

Now let us consider other subtrees attached to $r_1$ in $T_I'$. We divide $T_I'\backslash\bigvee\limits_{i=1}^{\alpha_1}T_i$, the remainder of $\bigvee\limits_{i=1}^{\alpha_1}T_i$ in $T_I'$, into four subtrees $T_{A_1}\bigvee T_{B_1}\bigvee T_{C_1}\bigvee T_{D_1}$, such that in $T_I$ the subtree $T_{A_1}\bigvee T_{C_1}$ is attached to $u_1$ and $T_{B_1}\bigvee T_{D_1}$ is attached to $r_1$. Since $T_I'$ and $T_{II}'$ differ by one permutation move, and the rest children of $r_1$ are not involved in the permutation move, there are four subtrees corresponding to $T_{A_1}\bigvee T_{B_1}\bigvee T_{C_1}\bigvee T_{D_1}$ in $T_{II}'$, say $T_{A_2}\bigvee T_{B_2}\bigvee T_{C_2}\bigvee T_{D_2}$. Without loss of generality, we assume in $T_{II}$ the subtree $T_{A_2}\bigvee T_{D_2}$ is attached to $u_2$ and the other subtree $T_{B_2}\bigvee T_{C_2}$ is attached to $r_2$, see Figure \ref{figure12}. We remark that because of the permutation move, $T_{A_1}$ and $T_{A_2}$ may represent different rooted trees, but they have the same number of edges, i.e. $|E(T_{A_1})|=|E(T_{A_2})|$. This equality also holds for the other three pairs of subtrees.
\begin{figure}
\centering
\includegraphics{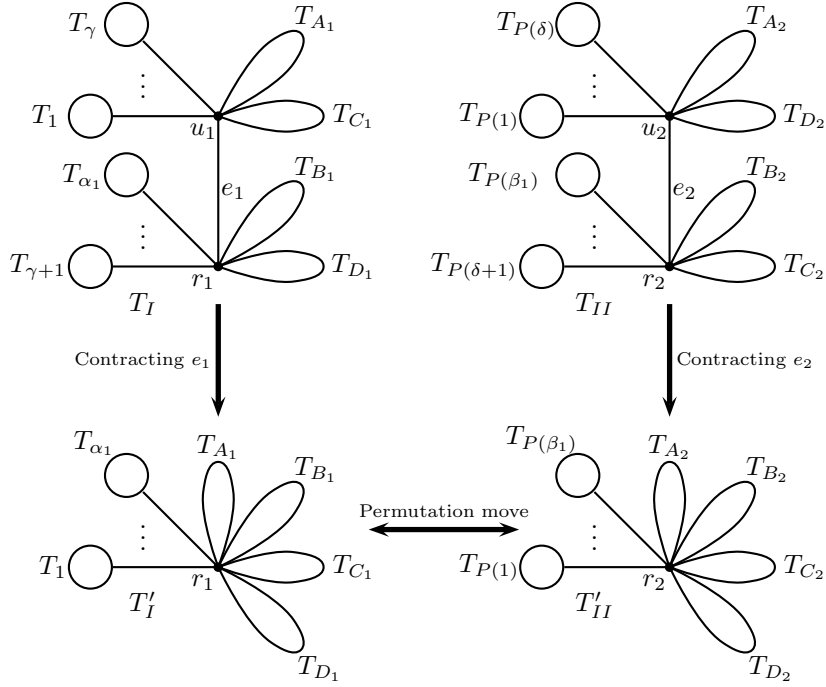}\\
\caption{The relations between $T_I, T_{II}, T_I', T_{II}'$}\label{figure12}
\end{figure}

According to our selection, we have $d(u_1)=d(u_2)$. On the other hand, we know that $d(u_1)=\sum\limits_{i=1}^{\gamma}|E(T_i)|+|E(T_{A_1})|+|E(T_{C_1})|$ and $d(u_2)=\sum\limits_{i=1}^{\delta}|E(T_{P(i)})|+|E(T_{A_2})|+|E(T_{D_2})|$. Together with $|E(T_{A_1})|=|E(T_{A_2})|$, $|E(T_{D_1})|=|E(T_{D_2})|$ and $\sum\limits_{i=1}^{\alpha_1}|E(T_i)|=\sum\limits_{i=1}^{\beta_1}|E(T_{P(i)})|$ (by the definition of the permutation move), we deduce that
$\sum\limits_{i=1}^{\gamma}|E(T_i)|+|E(T_{C_1})|=\sum\limits_{i=1}^{\delta}|E(T_{P(i)})|+|E(T_{D_1})|$,
and $\sum\limits_{i=\gamma+1}^{\alpha_1}|E(T_i)|+|E(T_{D_1})|=\sum\limits_{i=\delta+1}^{\beta_1}|E(T_{P(i)})|+|E(T_{C_1})|$.

Consider the following permutation on $T_I$, see Figure \ref{figure8}. Roughly speaking, this permutation move is obtained from the permutation move on $T_I'$ by splitting $v_1$ into $u_1$ and $r_1$, and all other subtrees $T_{\alpha_1+1}, \cdots, T_{\alpha_n}$ on $v_2, \cdots, v_n$ are the same as that of the permutation move on $T_I'$. Since $\sum\limits_{i=1}^{\gamma}|E(T_i)|+|E(T_{C_1})|=\sum\limits_{i=1}^{\delta}|E(T_{P(i)})|+|E(T_{D_1})|$  and $\sum\limits_{i=\gamma+1}^{\alpha_1}|E(T_i)|+|E(T_{D_1})|=\sum\limits_{i=\delta+1}^{\beta_1}|E(T_{P(i)})|+|E(T_{C_1})|$, we find that this move satisfies the definition of permutation move. It is not difficult to find that under this permutation move $T_I$ will be transformed into $T_{II}$.
\begin{figure}
\centering
\includegraphics{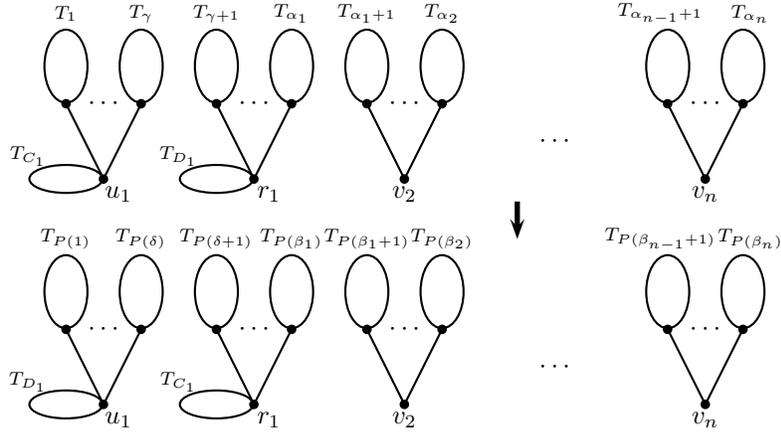}\\
\caption{A permutation move connecting $T_I$ and $T_{II}$}\label{figure8}
\end{figure}

If $v_i\neq r_1$ for any $1\leq i\leq n$, the proof is similar as above. One just needs to notice that in this case $\alpha_1=\beta_1=0$, in other words, $\bigvee\limits_{i=1}^{\alpha_1}T_i=\bigvee\limits_{i=1}^{\beta_1}T_{P(i)}=\emptyset$. The proof is finished.
\end{proof}

\begin{corollary}
Let $T_I$ and $T_{II}$ be two reduced rooted trees, then $D(T_I)=D(T_{II})$ if and only if they can be connected by one permutation move.
\end{corollary}

As an example, we will show how to use a permutation move to connect the two rooted trees described in Figure \ref{figure5}. Choose three vertices $v_1, v_2$ and $v_3$ from $T_1$ (see Figure \ref{figure5}, notice that here $v_1$ is a descendant of $v_2$) and $P=(6 3 2 5 4 1)=(1 6)(2 3)(4 5)\in \mathcal{S}_6$. Now the following permutation move (see Figure \ref{figure9}) can be used to transform $T_1$ into $T_2$. Note that we have proved $T_1, T_2$ cannot be connected by exchange moves.
\begin{figure}
\centering
\includegraphics{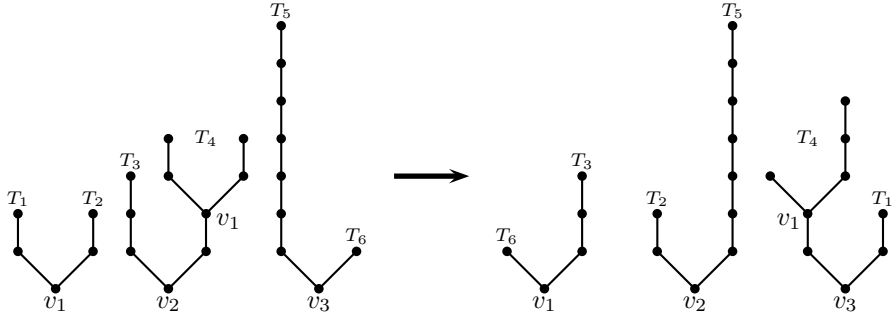}\\
\caption{A permutation move connecting $T_1$ and $T_2$ in Figure \ref{figure5}}\label{figure9}
\end{figure}

Finally, we would like to remark that if two rooted trees can be connected by a permutation move, in general the permutation move is not unique.

\section{The realization problem revisited}
In this section we revisit the first question mentioned in Section 1, i.e. for a given polynomial $f(q)\in\mathds{Z}[q]$, can we find a rooted tree $T$ such that $Q(T)=f(q)$? According to Proposition 2.4, it is equivalent to ask the following question from the viewpoint of graph theory.
\begin{question}
For a given $n$-tuple $\{a_1, a_2, \cdots, a_n\}$ (without loss of generality, we suppose $0\leq a_1\leq a_2\leq\cdots\leq a_n$), is there a rooted tree $T$ with $D(T)=\{a_1, a_2, \cdots, a_n\}$?
\end{question}

It is easy to find some obstacles for a $n$-tuple $\{a_1, \cdots, a_n\}$ being the set $D(T)$ of some rooted tree $T$. For example, it is evident to see that $a_n=n-1$ and $a_{n-1}<a_n$. Hence both $\{0, 0, 1, 1, 2\}$ and $\{0, 0, 1, 2, 2\}$ are not realizable. On the other hand, $a_1$ must be 0. Actually, it is easy to observe that the number of 0's in $\{a_1, \cdots, a_n\}$ is not less than the number of 1's, since each vertex $v$ with $d(v)=1$ has a child $w$ with $d(w)=0$. Therefore, for instance $\{0, 1, 1, 3, 4\}$ is also not realizable. However, even if $\{a_1, \cdots, a_n\}$ satisfies all conditions above, it is not always realizable. As an example, one can easily find that although $\{0, 1, 2, 2, 4\}$ satisfies all conditions above, it cannot be realized as the set $D(T)$ for some rooted tree $T$.

If there exists a rooted tree $T$ with $D(T)=\{a_1, \cdots, a_n\}$, then we know that the plucking polynomial of $T$ equals $\frac{[a_n]_q!}{\prod\limits_{i=1}^{n-1}[a_i+1]_q}$. Since the plucking polynomial of a rooted tree can be written as the product of some $q$-binomial coefficients, it follows that $\frac{[a_n]_q!}{\prod\limits_{i=1}^{n-1}[a_i+1]_q}$ can be written as the product of some $q$-binomial coefficients. As the main result of this section, we will show that this condition is not only necessary but also sufficient. Hence it offers a complete answer to Question 5.1.
\begin{theorem}
For a given $n$-tuple $\{a_1, a_2, \cdots, a_n\}$ where $0\leq a_1\leq a_2\leq\cdots\leq a_n$, there exists a rooted tree $T$ such that $D(T)=\{a_1, a_2, \cdots, a_n\}$ if and only if $\frac{[a_n]_q!}{\prod\limits_{i=1}^{n-1}[a_i+1]_q}$ can be written as the product of some $q$-binomial coefficients.
\end{theorem}

Before giving the proof of this theorem, let us take a brief review of our main result in \cite{Che2017}. Assume that a polynomial $f(q)$ is a product of some $q$-binomial coefficients, i.e.
\begin{center}
$f(q)=\prod\limits_{i=1}^{k}\binom{m_i+n_i}{m_i,n_i}_q=\prod\limits_{i=1}^{k}\frac{[m_i+n_i]_q!}{[m_i]_q![n_i]_q!}$.
\end{center}
If a $q$-number $[p]_q$ appears both in the numerator and denominator of $f(q)$, then we delete both of them. Finally we will obtain a fraction $\frac{[a_1]_q\cdots[a_s]_q}{[b_1]_q\cdots[b_t]_q}$ and $a_i\neq b_j$. We call this fraction the \emph{reduced form} of $f(q)$. It is not difficult to observe that the reduced form is unique.
\begin{theorem}[\cite{Che2017}]
Consider a product of $q$-binomial coefficients $f(q)=\prod\limits_{i=1}^{k}\binom{m_i+n_i}{m_i,n_i}_q$, then $f(q)$ can be realized as the plucking polynomial of some rooted trees if and only if the numerator of the reduced form of $f(q)$ does not repeat any $q$-number.
\end{theorem}

We would like to remark that we can always find a binary rooted tree $T$ to realize $f(q)$, and $\prod\limits_{i=1}^{k}\binom{m_i+n_i}{m_i,n_i}_q$ coincides with the state product formula (Proposition 2.3) of the plucking polynomial of $T$ if we ignore the contributions from those vertices that have only one child (note that these contributions are trivial). The readers are referred to \cite{Che2017} for more details.

Now we give the proof of Theorem 5.2.
\begin{proof}
The ``only if" part has been explained in Section 2, therefore it suffices to prove the``if" part.

First note that if $a_{n-1}=a_n$ then obviously $\frac{[a_n]_q!}{\prod\limits_{i=1}^{n-1}[a_i+1]_q}$ cannot be written as a product of some $q$-binomial coefficients. We claim that actually we can assume that $a_{n-1}\leq a_n-2$. This is because, if $a_{n-1}=a_n-1$ then we have
\begin{center}
$\frac{[a_n]_q!}{\prod\limits_{i=1}^{n-1}[a_i+1]_q}=\frac{[a_{n-1}]_q!}{\prod\limits_{i=1}^{n-2}[a_i+1]_q}$.
\end{center}
If $\{a_1, a_2, \cdots, a_{n-1}\}$ can be realized as the set $D(T)$ for some rooted tree $T$, then we can find a rooted tress $T'$ realizing $\{a_1, a_2, \cdots, a_n\}$ by applying a stabilization on $T$. Hence from now on let us assume that $a_{n-1}\leq a_n-2$.

Note that if $\frac{[a_n]_q!}{\prod\limits_{i=1}^{n-1}[a_i+1]_q}$ can be written as the product of some $q$-binomial coefficients $\prod\limits_{i=1}^{k}\binom{m_i+n_i}{m_i,n_i}_q$, then the numerator of the reduced form does not repeat any $q$-number, since the numerator equals $[a_n]_q!=[1]_q[2]_q\cdots[a_n]_q$. According to Theorem 5.3 we know that there exists a rooted tree $T$ such that $Q(T)=\prod\limits_{i=1}^{k}\binom{m_i+n_i}{m_i,n_i}_q=\frac{[a_n]_q!}{\prod\limits_{i=1}^{n-1}[a_i+1]_q}$. With some destabilizations we can assume that $T$ is reduced.

Let us assume that $D(T)=\{b_1, \cdots, b_m\}$ and $0\leq b_1\leq b_2\leq\cdots\leq b_m$. Then we have
\begin{center}
$\frac{[b_m]_q!}{\prod\limits_{i=1}^{m-1}[b_i+1]_q}=\frac{[a_n]_q!}{\prod\limits_{i=1}^{n-1}[a_i+1]_q}$.
\end{center}
Since $a_{n-1}\leq a_n-2$ and $b_{m-1}\leq b_m-2$, it follows that $b_m=a_n$. Now we have
\begin{center}
$\frac{1}{\prod\limits_{i=1}^{m-1}[b_i+1]_q}=\frac{1}{\prod\limits_{i=1}^{n-1}[a_i+1]_q}$,
\end{center}
which implies $\{b_1+1, \cdots, b_{m-1}+1\}=\{a_1+1, \cdots, a_{n-1}+1\}$ and $m=n$. Therefore $D(T)=\{b_1, \cdots, b_m\}=\{a_1, \cdots, a_n\}$, this completes the proof.
\end{proof}

\section*{Acknowledgements}
The first author would like to thank Hao Zheng for helpful discussion. The first author is supported by NSFC 11301028 and NSFC 11571038. The third author is supported by the Simons Foundation Collaboration Grant for Mathematicians--316446 and CCAS Dean's Research Chair award.

\bibliographystyle{amsplain}

\end{document}